\documentclass[letterpaper,12pt]{article}
\usepackage{amsmath,amsthm,amssymb}
\usepackage{graphicx}
\usepackage{subfigure}
\usepackage{cite}

\usepackage{hyperref}

\def\RR{\mathbb R}

\newtheorem{theorem}{Theorem}
\newtheorem{lemma}{Lemma}
\theoremstyle{definition}
\newtheorem{definition}{Definition}
\theoremstyle{remark}

\begin{document}
\rmfamily
\title{Non-existence of polyhedral immersions of triangulated surfaces in $\RR^3$}
\author{Undine Leopold}

\date{\today}
\maketitle
\begin{abstract}
\noindent We present and apply a method for disproving the existence of polyhedral immersions in $\RR^3$ of certain triangulations on non-orientable surfaces. In particular, it is proved that neither of the two vertex-minimal, neighborly $9$-vertex triangulations of the non-orientable surface of genus $5$ are realizable as immersed polyhedral surfaces in $\RR^3$.
\\[0.5ex]
\end{abstract}

%
%
%
%
\section{Introduction and basic notions}\label{sec:1}

A \emph{closed surface} (without boundary) is a compact Hausdorff space which is locally homeomorphic to $\RR^2$.
For the purposes of this article, we abbreviate the term to \emph{surface} and also require connectedness (which is the usual assumption in this context). Surfaces are classified either as the orientable, $g$-fold connected sum of tori, denoted $M_g$ (sphere with $g$ handles), or the non-orientable, $h$-fold connected sum of projective planes, denoted $N_h$ (sphere with $h$ cross-caps). The orientable \emph{genus} $g$ relates to the Euler characteristic $\chi$ of an orientable surface by $\chi=2-2g$, whereas in the non-orientable case we have $\chi=2-h$ for the non-orientable \emph{genus} $h$.  A \emph{triangulation of a (closed) surface  $S$} is a two-dimensional simplicial complex $\Delta$ whose underlying topological space $|\Delta|$ is homeomorphic to $S$.

It is well-known that orientable surfaces can be smoothly embedded into~$\RR^{3}$, whereas non-orientable surfaces may only be smoothly immersed. 
Due to Steinitz' theorem \cite{s1922:pur, z1995:lop}, every polyhedral map on the sphere can be realized as the boundary complex of a convex $3$-polytope. Archdeacon, Bonnington, and Ellis-Monaghan \cite{abem2007:htetmis} proved that every toroidal map can be \emph{exhibited} in Euclidean $3$-space, with the implication that every triangulated torus is geometrically realizable with flat triangular faces and without self-intersection. For other triangulated orientable surfaces, an analogous general result on the existence of such realizations is not possible; the first counterexample was presented in \cite{bgo2000:otgoom}, and Schewe's result \cite{s2010:nmvtossnuomass} allows the construction of counterexamples for all orientable genera greater than or equal to $5$.  

The realizability question may be extended to triangulations of non-orientable surfaces in the following way. Let $\Delta$ be a triangulation of a surface $S=|\Delta|$, and let $V$ be the set of vertices. Each assignment $\psi:V\rightarrow\RR^3$ of coordinates to the set of vertices induces a simplex-wise linear map $\phi_\psi\colon\left|\Delta\right|\rightarrow\RR^3$. Each simplexwise linear map $\phi_\psi\colon\left|\Delta\right|\rightarrow\RR^3$ is completely determined by its restriction $\psi\colon V\rightarrow\RR^3$ on the set of vertices. We say that $\phi_\psi$ is a \emph{polyhedral realization} or \emph{geometric realization} of $\Delta$ (or $S$) if it is a simplex-wise linear immersion (i.e., it is locally injective). In addition, for triangulations of \emph{orientable} surfaces $S$, $\phi_\psi$ must meet the stricter condition of being an embedding (globally injective). We also use the terms \emph{polyhedral immersion} and \emph{polyhedral embedding} to distinguish among these cases. For either case, local injectivity assures that $i$-dimensional faces of $\Delta$ are mapped to flat, $i$-dimensional simplices in $\RR^3$ (points, line segments, triangles). 

Due to Brehm's result \cite{b1983:antms}, any non-orientable surface possesses a triangulation which is not geometrically realizable. Thus in general, for orientable surfaces of genus $g\geq 5$ and non-orientable surfaces alike, realizability of a triangulation must be decided on a case-by-case basis. Successful realizations of surfaces in $\RR^3$ were obtained mainly by hand or by using heuristic algorithms, see \cite{c1949:apwd,b1981:pmzevgd,mcmsw1983:p2mie3wulg,b1987:amspog3w10v,bb1987:anpog3w10v,b1987:msprodrm,b1989:agrwsidefdrm,bb1989:apog4wmnovams,b1990:htbmpmotbs,c1994:vmsiotkbits,l2008:earrots,b2008:ohmffros,hlz2010:srwtisf,bl2016:peaiomt2mwfv}. 

 A particularly interesting set of test cases are those triangulations with a high degree of connectivity among the vertices, i.e., 
triangulations with the minimal number of vertices for the given genus, or realizations of \emph{neighborly} triangulations of surfaces. A triangulation is \emph{neighborly} when the edge graph is complete, i.e., there is an edge connecting any pair of vertices. Triangulations of a surface of Euler characteristic $\chi$ with $n$ vertices exist when the Heawood bound \cite{h1890:mct} is satisfied, i.e., when
\[n\geq \left\lceil\frac{7+\sqrt{49-24\chi}}{2}\right\rceil,\]
except in the case of the surfaces $M_2$, $N_2$, and $N_3$, where no corresponding triangulations exist, as shown in \cite{r1955:wmdgnfimwdzk, jr1980:mtoos} (there exist triangulations with one additional vertex).
 Neighborly triangulations of surfaces exist whenever $n$ equals the unrounded expression on the right hand side of the Heawood bound, i.e. for Euler characteristic $\chi=2,1,0,-3,-5,-10,\ldots$ and corresponding number of vertices $n=4,6,7,9,10,12,\ldots$.

Bokowski and Sturmfels \cite{bs1989:csg}
have proposed algorithms for deciding the realizability of triangulations and other geometric structures, but the decision process remains computationally difficult and non-realizability results for surfaces have been limited to specific cases of orientable surfaces, e.g. \cite{bgo2000:otgoom, s2010:nmvtossnuomass}. This article investigates non-realizability for triangulations of non-orientable surfaces exclusively, using the expected self-intersection in $\RR^3$ as a starting point. It appears that this self-intersection has not been studied in greater detail, aside from an article by Cervone \cite{c1994:vmsiotkbits}. The aim of this article is to show the usefulness of such an investigation for triangulations with few vertices.

Following a motivating example at the end of Section \ref{sec:2}, the main result is presented in Section \ref{sec:3}; neither of the two vertex-minimal $9$-vertex triangulations of the non-orientable surface of genus $5$ are polyhedrally immersible in $\RR^3$. The proof first appeared in the author's Diploma Thesis \cite{l2009:peaiot2m}, and the result is now published as an article for the first time. Section \ref{sec:4} concludes this paper with some additional remarks.

\section{Properties of immersions}\label{sec:2}

Let $\phi\colon S\rightarrow\RR^3$ be an immersion of a (closed) non-orientable surface $S$ into~$\RR^3$. A point $p\in \phi(S)$ is a \emph{point of self-intersection} if $\phi^{-1}(p)$ consists of more than one point. 
If  $\phi^{-1}(p)$ consists of precisely two, three, or $k$ points, $p$ is called a \emph{double, triple, or $k$-fold point}. The number $\mathrm{card}\left(\phi^{-1}(p)\right)>1$ is called the \emph{order} of the point of self-intersection. The set of all points of self-intersection of $\phi(S)$ in $\RR^3$ is called the \emph{set of self-intersections},  or \emph{double set}, of the immersion and is denoted $D_\phi$.  
The pre-image of $D_\phi$ is called the \emph{singular set} in $S$ and is denoted $\phi^{-1}(D_\phi)$.

\emph{Genericity} or \emph{general position} for an immersion $\phi\colon S\rightarrow\RR^3$ is given when all self-intersections are transversal, at most three disks intersect in a point, which occurs at finitely many points in $\RR^3$, and at these triple points the disks intersect topologically in the way coordinate planes do. For a general position immersion both the double set $D_\phi$ and the singular set $\phi^{-1}(D_\phi)$ consist of the union of closed curves. Furthermore, the double set cannot contain isolated points of $\RR^3$, and the singular set cannot contain isolated points of $S$.

\subsection{Properties of the double set and singular set}

Let $\Delta$ be a triangulation of $S$, and let $V$ be the set of vertices of $\Delta$. Let $\phi_\psi\colon \left|\Delta\right|\rightarrow\RR^3$ be a polyhedral realization (simplex-wise linear immersion) in $\RR^3$ induced by a coordinatization $\psi\colon V\rightarrow \RR^3$ of the vertex set $V$.  
It is possible to perturb the image $\psi(v)$ of each vertex $v\in V$ within a ball of some radius $\varepsilon>0$, such that the induced $\phi_\psi$ is still a polyhedral realization.
With such a perturbation, several undesirable situations may be resolved.
First, since $\Delta$ consists of finitely many triangles and $\phi_\psi$ is simplex-wise linear, the set of intersection points of order three and higher may be dissolved into a discrete (and therefore, finite) set of triple points connected by curve segments composed of double points. 
Thus, as long as no geometric symmetry is required, the existence of a polyhedral realization of a triangulation $\Delta$ of $S$ implies that there exists also a polyhedral realization in general position. 

Second, we may apply another perturbation to shift the double set and singular set according to our preferences. For example, we may move triple points and most double points away from the vertices and edges, obtaining the following property which is minimally stronger than general position of the surface.
\begin{definition}\label{def:1}
We call a polyhedral realization $\phi_\psi\colon \left|\Delta\right|\rightarrow\RR^3$ \emph{polyhedrally generic} if it is in general position and both of the following additional conditions are satisfied:\\
(1) the images of any edge $ab$ and any triangle $cde$, for distinct vertices $a$, $b$, $c$, $d$, $e$, are either disjoint or intersect only in their relative interiors,\\
(2) any triple points lie in the relative interiors of the three (images of) triangle faces which intersect there. 
\end{definition}
A polyhedrally generic immersion exists for any polyhedrally realizable triangulation $\Delta$.

A staple argument from topology  (see also \cite[Lemma on p. 411]{b1974:tpasois} for one variant) is that, generically, a simple closed curve in $\mathbb{R}^3$ pierces an immersed surface in general position (orientable or non-orientable) in an even number of points transversally. 
Using this argument, we obtain two Lemmas. The first statement of Lemma \ref{lem:1} was already a Corollary in \cite{b1974:tpasois}.
\begin{lemma}\label{lem:1}
Let $\phi\colon S\rightarrow\RR^3$ be a general position immersion of a closed surface $S$ into $\RR^3$. Let $C$ be a simple closed orientation-reversing curve embedded in $S$, i.e., every tubular neighborhood of $C$ in $S$ contains a M\"obius strip. Then the closed curve $\phi(C)$ must meet $D_\phi \subseteq \RR^3$ in at least one point. Moreover, if $C$ meets the singular set $\phi^{-1}(D_\phi)$ in a finite number of points, and if all these intersections are transversal, and if furthermore $\phi(C)$ does not pass through any of the triple points of $\phi(S)$, then 
\[\mathrm{card}(C\cap\phi^{-1}(D_\phi))\] 
is odd. 
\end{lemma}

\begin{proof}
We remark that since $C$ is orientation-reversing, $S$ is necessarily a non-orientable surface. First, assume that $\phi(C)$ does not meet $(D_\phi)$. Then $\phi(C)$ is a \emph{simple} closed curve in $\RR^3$. Furthermore, there exists a non-orientable tubular neighborhood of $C$ which does not contain any points of $\phi^{-1}(D_\phi)$.  Its image under $\phi$ is an embedded M\"obius strip in $\RR^3$ and contains no points of $D_\phi$. Consequently, it is possible to slightly distort  $\phi(C)$ (i.e., homotope in a neighborhood)
 such that the resulting simple closed curve $\hat{C}$ pierces the embedded M\"obius strip exactly once transversally and contains no other points of $\phi(S)$. This is impossible (an even number of transversal intersections is required), and hence $C$ must meet $\phi^{-1}(D_\phi)$ at least once, and $\phi(C)$ must meet $D_\phi$ in at least one point.

Assume now that $C$ meets $\phi^{-1}(D_\phi)$ in a finite number of points, such that $\phi(C)$ does not contain triple points of $\phi(S)$. Assume further that all of these intersections of $C$ and $\phi^{-1}(D_\phi)$ are transversal. Then it is always possible to distort $C$ slightly such that the resulting simple closed, orientation-reversing curve $\tilde{C}$ has a simple closed image $\phi(\tilde{C})$ which avoids all triple points, while $\tilde{C}$ possesses the same (finite) number of (transversal) intersections with $\phi^{-1}(D_\phi)$ as $C$ (although not necessarily at the same points). In other words, $\tilde{C}$ maintains all properties of $C$ and just avoids running through pairs of points in the singular set which are pre-images of the same double point in $\RR^3$. Then, $\mathrm{card}\left(\tilde{C}\cap \phi^{-1}(D_\phi)\right)=\mathrm{card}\left(\phi(\tilde{C})\cap D_\phi\right)$.

By definition of $C$ and $\tilde{C}$, there exists a non-orientable tubular neighborhood $T(\tilde{C})$ of $\tilde{C}$ in $S$ such that its image $\phi(T(\tilde{C}))$ is an \emph{embedded} M\"obius strip. The simple closed curve $\phi(\tilde{C})$ can be distorted slightly in such a way that it is lifted off the M\"obius strip, except for one point where it pierces the strip transversally. 
Furthermore, the lift can be done in such a fashion that each point in $\phi(\tilde{C})\cap D_\phi$ corresponds to exactly one point of transversal intersection of the lifted curve $\hat{C}$ with $\phi(S)\setminus \phi(T(\tilde{C}))$, i.e. $\mathrm{card}(\phi(\tilde{C})\cap D_\phi)= \mathrm{card}(\hat{C}\cap \phi(S))-1$, and 
 $\hat{C}\cap D_{\phi}=\emptyset$.
 
Now, the number of transversal intersections of $\hat{C}$ and $\phi(S)$ is even and, by the considerations above, equal to the number of (transversal) intersections of $\tilde{C}$ with $\phi^{-1}(D_\phi)$ plus one. Therefore \[\mathrm{card}\left(\tilde{C}\cap\phi^{-1}(D_\phi)\right)=\mathrm{card}\left(C\cap\phi^{-1}(D_\phi)\right)\] must be odd.
\end{proof}
In a second Lemma, the corresponding statement for simple closed orientation-preserving curves embedded in $S$, i.e., curves possessing an orientable tubular neighborhood, is presented.
\begin{lemma}\label{lem:2}
Let $\phi\colon S\rightarrow\RR^3$ be a general position immersion of a closed surface $S$ into $\RR^3$. Let $C$ be a simple closed orientation-preserving curve embedded in $S$. If $C$ meets the singular set $\phi^{-1}(D_\phi)$ in a finite number of points, and if all these intersections are transversal, and if furthermore $\phi(C)$ does not pass through any of the triple points of $\phi(S)$, then 
\[\mathrm{card}(C\cap\phi^{-1}(D_\phi))\] 
is even.
\end{lemma}

\begin{proof}
Analogous.
\end{proof}

\subsection{Polyhedral immersions of triangulations of non-orientable surfaces}

Triangulations are simplicial complexes giving \emph{polyhedral} decompositions of closed surfaces, in the sense that no points on the boundary of any simplex (face) are identified, and the intersection of any two simplices (faces) is again a simplex (face) of the triangulation, possibly the empty simplex. Consequently, triangulations can be given as a list of triangles $abc$, $def$, etc., with vertices $a$, $b$, $c$, $d$, $e$, $f$, etc. For background on more general \emph{polyhedral maps} and their realizations see \cite{bs1997:pm}.

We denote with $\phi_{\psi}(a)$ the image of vertex $a$, with $\phi_{\psi}(ab)$ the image of edge $ab$, and $\phi_{\psi}(abc)$ denotes the image of triangle $abc$. For convenience, we omit the map $\phi_{\psi}$ in most instances. The term \emph{immersed vertex (edge, triangle)} is used as shorthand for the image of a vertex (edge, triangle) under an immersion $\phi_{\psi}$, wherever additional clarification is necessary. 

\begin{figure}[bth]
        \centering
        {\includegraphics[width=0.5\linewidth]{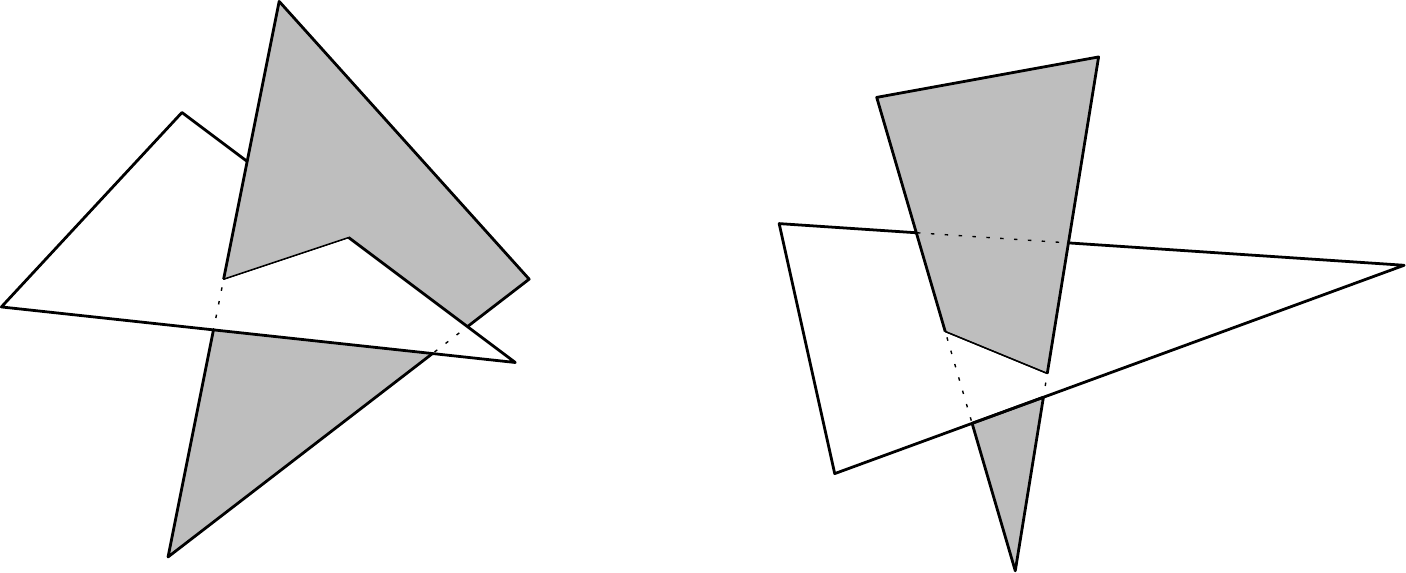}}
        \caption{How triangles intersect in a polyhedrally generic realization.}
        \label{fig:1}
\end{figure}

\begin{paragraph}{Topological methods in the discrete setting.}
Consider a triangulation~$\Delta$ of a non-orientable closed surface $S$ with vertex set $V$. By the remarks preceding Definition \ref{def:1}, we may assume that there exists a polyhedrally generic realization (immersion) $\phi_{\psi}\colon \left|\Delta\right|\rightarrow \RR^3$ into $\RR^3$ induced by a coordinatization $\psi\colon V\rightarrow \RR^3$ of the vertices.

No two immersed triangles with a shared vertex or edge in $\Delta$ can intersect other than at the common vertex or edge, because that would produce a contradiction to the local injectivity of the immersion $\phi_\psi$. Consequently, vertex-disjoint immersed triangles in $\RR^3$ can only intersect in one of the two ways pictured in Figure \ref{fig:1}. Note that of the six combined edges of both triangles, precisely two must pierce the respective other triangle. 
Furthermore, both the double set of $\phi_{\psi}$ in $\RR^3$ and the singular set in $S$ consist of polygonal curve segments which do not pass through the (immersed) vertices. 
By Definition \ref{def:1}, each edge in $\Delta$ possesses a finite number of intersections with the singular set, 
and all these intersections are transversal. 

These properties make simple edge cycles of the triangulation ideal objects for the application of the Lemmas~\ref{lem:1} and~\ref{lem:2}. If we find a simple orientation-preserving cycle $C$ in the edge graph of $\Delta$, i.e., a cycle possessing an orientable tubular neighborhood in $|\Delta|$, the number of intersections with the singular set has to be even in order to comply with polyhedral immersibility. By contrast, if $C$ is orientation-reversing, i.e., each tubular neighborhood contains a M\"obius strip, then the number of intersections with 
the singular set has to be odd and, particularly, positive.

\end{paragraph}

\begin{figure}[htb]
        \centering
        {\includegraphics[width=0.5\linewidth]{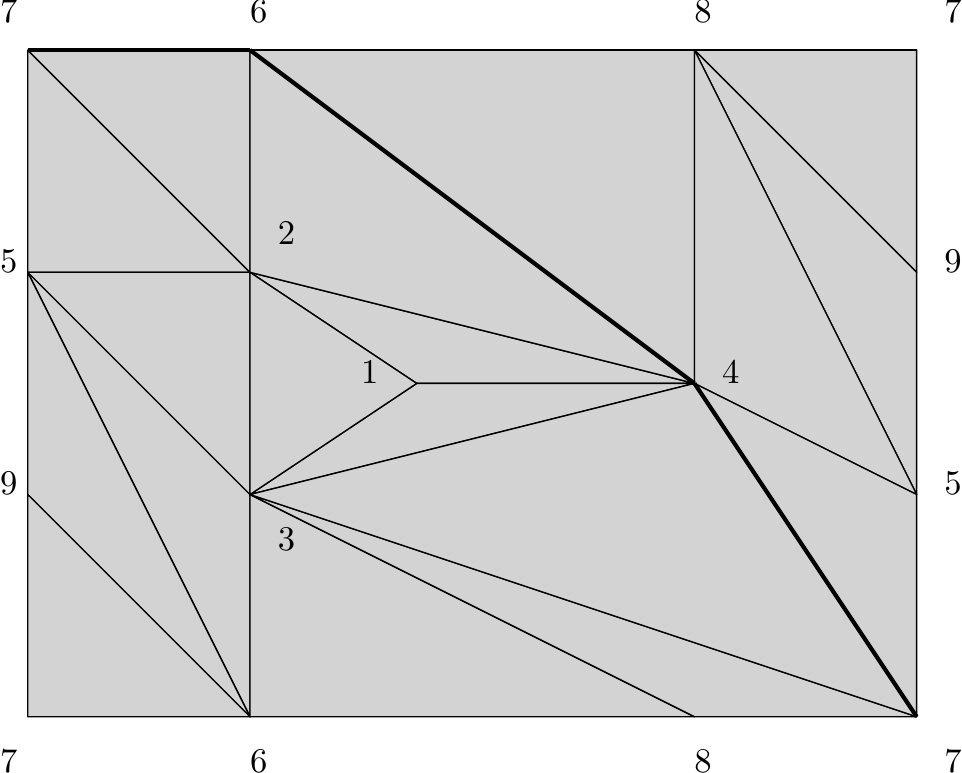}}
        \caption{A polyhedrally non-realizable triangulation of the Klein Bottle $N_2$.}
        \label{fig:2}
\end{figure}

\begin{paragraph}{A Klein bottle example.}
The triangulation in Figure \ref{fig:2} is one out of $187$ possible triangulations of the Klein bottle with $9$ vertices. Specifically, it is the eighth triangulation in the file provided on Lutz's Manifold Page \cite{lutz:manifoldpage}. We begin by assuming that a polyhedral realization, without loss of generality a polyhedrally generic realization, exists. We see that the curve underlying the edge cycle $67-46-47$ in a polyhedral immersion, marked in bold in Figure \ref{fig:2}, is orientation-reversing. By Lemma~\ref{lem:1}, this curve must meet the singular set an odd number of times. Observe that edge $46$ may not pierce any triangle in an immersion, since triangles sharing vertices may not intersect, and every triangle in the triangulation shares a vertex with either triangle $246$ or triangle $468$. Moreover, edge $47$ may not pierce any triangle, since every triangle in the triangulation shares a vertex with either triangle $347$ or triangle $457$. Considering how triangles intersect in a polyhedrally generic immersion (see Figure \ref{fig:1}), triangles $679$ and $458$ may not intersect, as edges $69$, $79$, $45$, $48$, $58$ may not pierce the respective other triangle (for similar reasons as before). Consequently, edge $67$ also cannot pierce triangle $458$. Next, triangles $257$ and $134$ may not intersect, as edges $25$, $57$, $13$, $14$, $34$ cannot pierce the respective other triangle. Therefore edge $27$ may not pierce triangle $134$. However, this also means that $267$ and $134$ may not intersect, as edges $13$, $14$, $34$, $26$, and $27$ cannot pierce the respective other triangle. This implies that $67$ also cannot pierce $134$.  Now, since all triangles besides $134$ and $458$ share a vertex with either $267$ or $679$, and we have shown that $67$ pierces neither $134$ nor $458$, edge $67$ cannot pierce any triangle and the cycle $67-46-47$ does not meet the singular set. This contradiction proves that the triangulation is not polyhedrally realizable.
\end{paragraph}

\section{Nine vertices do not suffice for a polyhedral realization of $N_5$}\label{sec:3}

The non-orientable surface $N_5$ of genus $5$ can be triangulated using $9$ vertices in precisely two combinatorially distinct ways, as listed on Lutz's \emph{Manifold Page} \cite{lutz:manifoldpage}. It cannot be triangulated with fewer vertices due to the Heawood bound \cite{h1890:mct}, see Section \ref{sec:1}.
The Euler characteristic of $\chi=-3$ and $|V|=9$ vertices require the triangulations to 
possess $36={9 \choose 2}$ edges and $24$ triangle faces. Observe that the edge graph is complete, i.e.,
 both possible triangulations are neighborly. Therefore, polyhedral immersibility should be particularly difficult to achieve. We will now prove (by contradiction) that it is indeed impossible.

\subsection{The first triangulation}\label{subsec:first}
\begin{figure}[htbp]
        \centering
        {\includegraphics[width=0.5\linewidth]{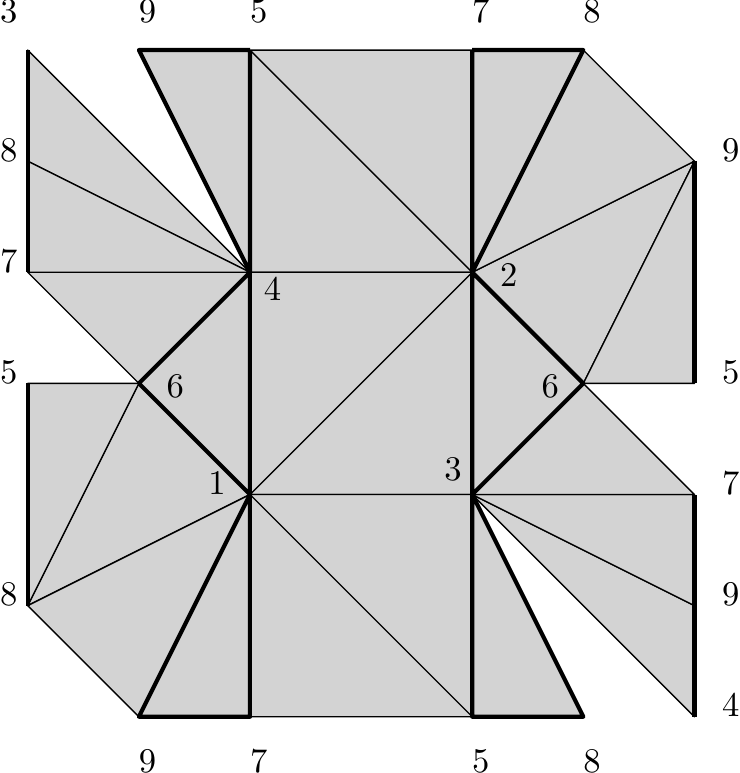}}
        \caption{The $9$-vertex-triangulation $\Delta_1$ of $N_5$.}
        \label{fig:3}
\end{figure}

\begin{figure}[htbp]
        \centering
        {\includegraphics[width=0.8\linewidth]{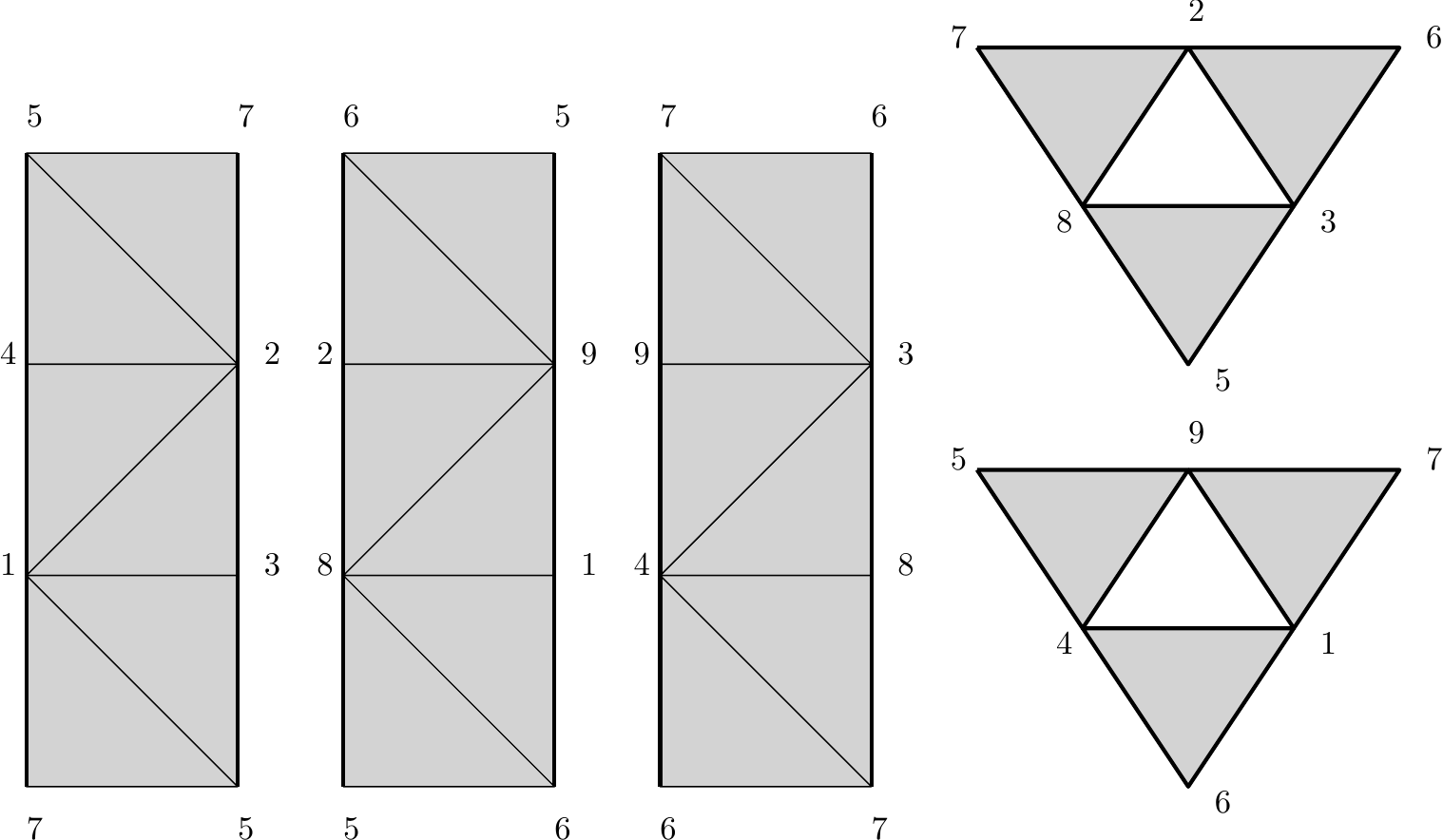}}
        \caption{$\Delta_1$ decomposes into three triangulated M\"obius strips and six remaining triangles.}
        \label{fig:4}
\end{figure}
Let $\Delta_1$ denote the first $9$-vertex triangulation of $N_5$, diagrammed in Figure \ref{fig:3}, and assume the existence of a polyhedral immersion of this triangulation. Then recall that, without loss of generality, we may choose a polyhedrally generic immersion $\phi_{\psi}\colon\left|\Delta_1\right|\rightarrow\RR^3$.

The triangulation $\Delta_1$ is composed of three triangulated M\"obius strips and six additional triangles, see Figure \ref{fig:4}. The automorphism induced by the permutation $\left(1 8 4 2 9 3\right)\left(5 6 7\right)$ on the vertices exchanges these M\"obius strips, and also exchanges the edges $56$, $57$, and $67$. Thus the three M\"obius strips are combinatorially equivalent and, also, edges $56$, $57$, and $67$ are combinatorially equivalent. The curve underlying the edge cycle $56-67-57$ is orientation-reversing in $\left|\Delta_1\right|$ (which can be derived from Figure \ref{fig:3}).  Lemma \ref{lem:1} implies that this curve meets 
the singular set an odd number of times. 

\subsubsection{The intersection table}

We introduce a so-called \emph{intersection table} as a useful tool for our proof. Consider the type of diagram depicted in Figure~\ref{fig:5}. This intersection table records intersections of edges and triangles, and therefore also records intersections of pairs of triangles. 

Triangle labels are noted on top of and to the left of the table, marking rows and columns. For triangles $abc$ and $def$, the box in the column marked $abc$ and row marked $def$ may contain the labels of their edges, namely $ab$, $bc$, $ac$, $de$, $ef$, and $df$. This notation only requires boxes above the main diagonal of the table (otherwise the information would be doubled). For better readability, the triangle labels marking the rows have been moved to the diagonal. 

The occurrence of edge labels in a box is decided based on a low-level obstruction to the intersection of edges and triangles, called \emph{edge-cut-analysis} by Cervone \cite{c1994:vmsiotkbits}. Specifically, for triangles $abc$ and $def$, edge $ab$ appears in the corresponding box if and only if triangles $abd$, $abe$, $abf$ are \emph{not} part of~$\Delta_1$; if one of these triangles is in $\Delta_1$, $ab$ can certainly not pierce $def$ without conflicting with the immersion property.
Empty boxes have been left out, and some rows and columns with empty boxes have been cropped entirely. For intersections pertaining to a specific triangle $abc$, the row marked $abc$ as well as the column marked $abc$ have to be consulted if they appear in the table. 

 \begin{figure}[htbp]
        \centering
        {\includegraphics[width=1.0\linewidth]{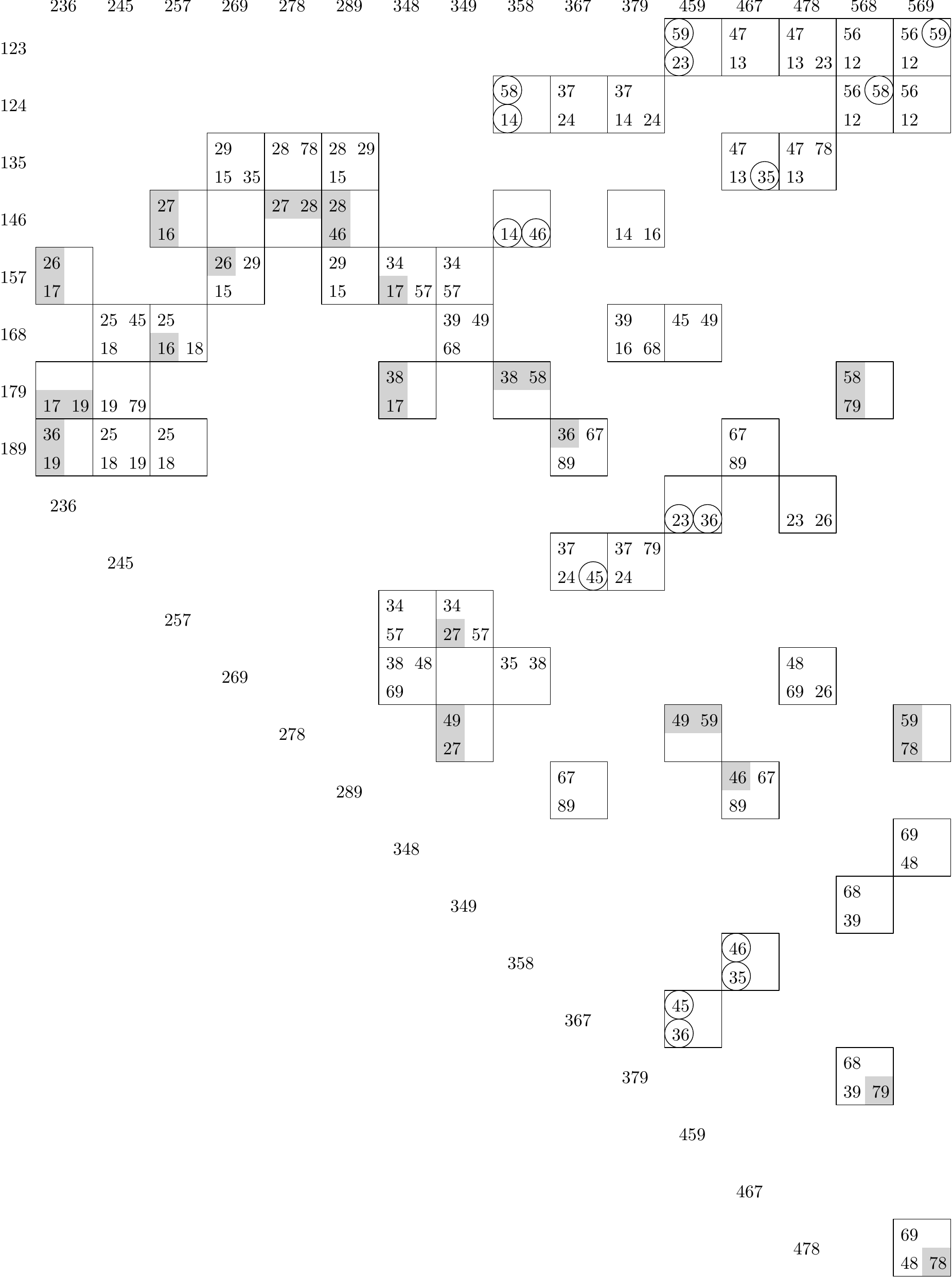}}
        \caption{Tracing the singular set in $|\Delta_1|$, first case.}
        \label{fig:5}
\end{figure}

Note that for the intersection table of $\Delta_1$ in Figure~\ref{fig:5} (and also in Figure~\ref{fig:8}), each edge label appears in precisely four boxes.
This means, in particular, that each of the edges $56$, $57$, and $67$ of the orientation-reversing cycle $56-67-57$ can meet 
the singular set at most two times. Therefore, either exactly one or each of these edges meets the singular set 
exactly once, in order to give an odd total of intersections of the edge cycle with the singular set. This gives rise to two cases, which we will examine subsequently. Before, however, we explain how to work with the intersection table. 

\begin{paragraph}{Tracing the course of the singular set.}
We attempt to trace the course of the singular set by circling or shading edge labels in the intersection table. 
An edge label is circled in a box if the edge pierces a triangle associated to the box, i.e., the triangle of which it is not a side, in a purported realization $\phi_\psi$ of the triangulation.  Considering how triangles intersect in space if $\phi_{\psi}$ is polyhedrally generic, see Figure \ref{fig:1}, exactly two or no edge labels at all have to be circled in each box. If it is clear that there is no intersection, the respective edge labels are shaded. In particular, if there is only one unshaded edge label in a box, the respective triangles may not intersect after all and we shade that label. If two edge labels in the same box are circled, then the remaining edge labels become shaded as the intersection of the associated triangles is determined completely. 

Recall that if an edge $ab$ pierces a triangle $efg$, both triangles incident with $ab$, say $abc$ and $abd$, intersect with $efg$. The diagram allows us to display these connections. Suppose we circle edge label $ab$ in the box corresponding to triangles $abc$ and $efg$. Then we also have to circle edge label $ab$ in the box corresponding to triangles $abd$ and $efg$ because we have an intersection of edge $ab$ and triangle $efg$. 

The same holds for the exclusion of intersections, marked by shading edge labels. Suppose edge label $ab$ is shaded in the box corresponding to triangles $abc$ and $efg$. This means that the intersection of edge $ab$ and triangle $efg$ is excluded. 
Consequently, edge label $ab$ must also be shaded in the box corresponding to triangles $abd$ and $efg$.  \end{paragraph}

\begin{figure}[htbp]
\centering
\includegraphics[width=0.2\linewidth]{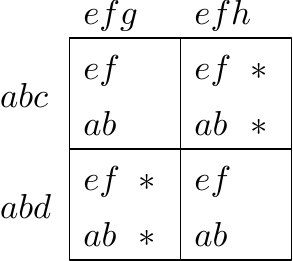}
\caption{Edges $ab$ and $ef$ are combinatorially coupled if they appear only in these boxes. The asterisks signify possible additional edge labels.}
\label{fig:6} 
\end{figure}

\begin{paragraph}{Combinatorially coupled edges.}
We shall pay special attention to the following situation. Consider two edges $ab$ and $ef$, and the triangles $abc$, $abd$, $efg$, $efh$ incident with them. Assume $ab$ and $ef$ appear in all four of the associated boxes, and in no other boxes. That is, if $ab$ pierces any triangle it must be $efg$ or $efh$, and if $ef$ pierces any triangle, it must be $abc$ or $abd$, and any triangle in $\Delta_1$ besides the four previously mentioned possesses a vertex in $\{a,b,c,d\}$ and a vertex in $\{e,f,g,h\}$. Assume further that the box associated to $abc$ and $efg$, as well as the box associated to $abd$ and $efh$, each only contain the labels $ab$ and $ef$, as in Figure \ref{fig:6}. The other two boxes may contain additional labels.
If edges $ab$ and $ef$ satisfy all of the above conditions, we call the pair $\left\{ab,ef\right\}$ \emph{combinatorially coupled}, since it is easily verified that $ab$ and $ef$ must possess the same number of (transversal) intersections with
the singular set. 
\end{paragraph}

\begin{figure}[htbp]
        \centering
        \subfigure[$12$ and $56$ meet the singular set exactly twice each.]
        {\includegraphics[width=0.2\linewidth]{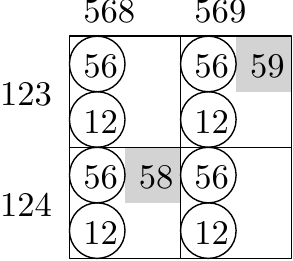}}\quad
        \subfigure[$12$ and $56$ do not meet the singular set.]
        {\includegraphics[width=0.2\linewidth]{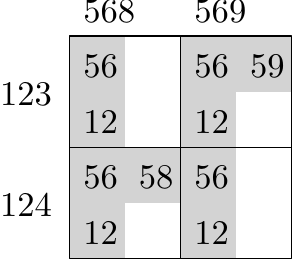}}\quad
        \subfigure[$12$ and $56$ meet the singular set exactly once each.]
        {\includegraphics[width=0.2\linewidth]{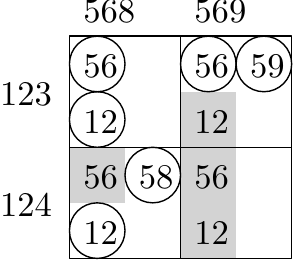}}\quad
        \subfigure[$12$ and $56$ meet the singular set exactly once each.]
        {\includegraphics[width=0.2\linewidth]{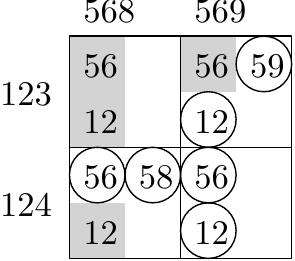}}
        \caption{Combinatorially coupled edges.}    
        \label{fig:7}
\end{figure}

\begin{paragraph}{An example of combinatorially coupled edges for $\Delta_1$.}
We illustrate combinatorially coupled edges with the concrete example of edges $56$ and $12$ of the triangulation $\Delta_1$ of $N_5$. 
Figure \ref{fig:7} shows the boxes of Figure \ref{fig:5} in which the labels $56$ and $12$ appear. Suppose edge $56$, common to triangle $568$ and $569$, pierces both (does not pierce either) triangles $123$ and $124$, i.e., it is circled (shaded) in all four boxes, then $12$ pierces both (does not pierce either) triangles $568$ and $569$, and vice versa, see Figure \ref{fig:7} (a) and (b). Then the number of intersections of $12$ and $56$ with the singular set is the same, namely two or zero. Note that in this case $59$ does not intersect triangle $123$, and that $58$ does not intersect triangle $124$ (both become shaded).

By contrast, suppose $56$ is only circled in the top (bottom) row, i.e. edge $56$ pierces triangle $123$ ($124$) but not triangle $124$ ($123$). Then edge $12$ must pierce triangle $568$ ($569$), whereas it must not pierce triangle $569$ ($568$), see Figure \ref{fig:7} (c) and (d). Thus both $12$ and $56$ intersect the singular set precisely once.  
Also, $58$ must intersect $124$ and $59$ must intersect $123$ (both become circled). 

Note that whether the additional edge labels $59$ and $58$ are circled or not circled in these boxes, respectively, depends solely on the parity of the number of intersections of $56$ and $12$ with the singular set. This is because $58$ and $59$ are the only additional edge labels in the displayed boxes. We use this property, which also occurs for the other combinatorially coupled pairs of $\Delta_1$, namely $\left\{57,34\right\}$ and $\left\{67,89\right\}$, in the following case distinction.

\end{paragraph}

\subsubsection{Case distinction}

\begin{paragraph}{First case.}
Without loss of generality (equivalence of edges $56$, $57$, $67$ under automorphism) we assume that edge $56$ meets 
the singular set exactly once, whereas edges $57$ and $67$ each meet the singular set 
an even number of times (i.e., exactly twice or not at all). We use the intersection table in Figure \ref{fig:5} to mark those intersections and non-intersections which we can derive for this case. Since edges $34$ and $89$ are coupled with $57$ and $67$, respectively, they each meet 
the singular set exactly twice or not at all, and edge $12$ meets 
the singular set exactly once, due to being coupled with edge $56$. From this we conclude that edge $59$ must pierce triangle $123$ and edge $58$ must pierce triangle $124$. Furthermore, there are several excluded intersections; edge $17$ and triangle $348$, edge $27$ and triangle $349$, edge $36$ and triangle $189$, as well as edge $46$ and triangle $289$, each have no point in common. 

Starting with this inital setup of shaded and circled edge labels, some consequences are derived in Figure \ref{fig:5} by the tracing process described earlier. Clearly, edges $58$ and $59$ have exactly one intersection with 
the singular set, while $89$ possesses an even number (zero or two). This means that cycle $58-89-59$ intersects the singular set an even number of times. 
 However, we can check that the curve underlying cycle $58-89-59$ is orientation-reversing (as it possesses no orientable neighborhood in $\left|\Delta_1\right|$) and thus should have an odd number of intersections with the singular set. 
 This is a contradiction to Lemma \ref{lem:1}, proving that this first case cannot occur in a polyhedral realization of the first triangulation.
\end{paragraph}

\begin{figure}[htbp]
        \centering
        {\includegraphics[width=1.0\linewidth]{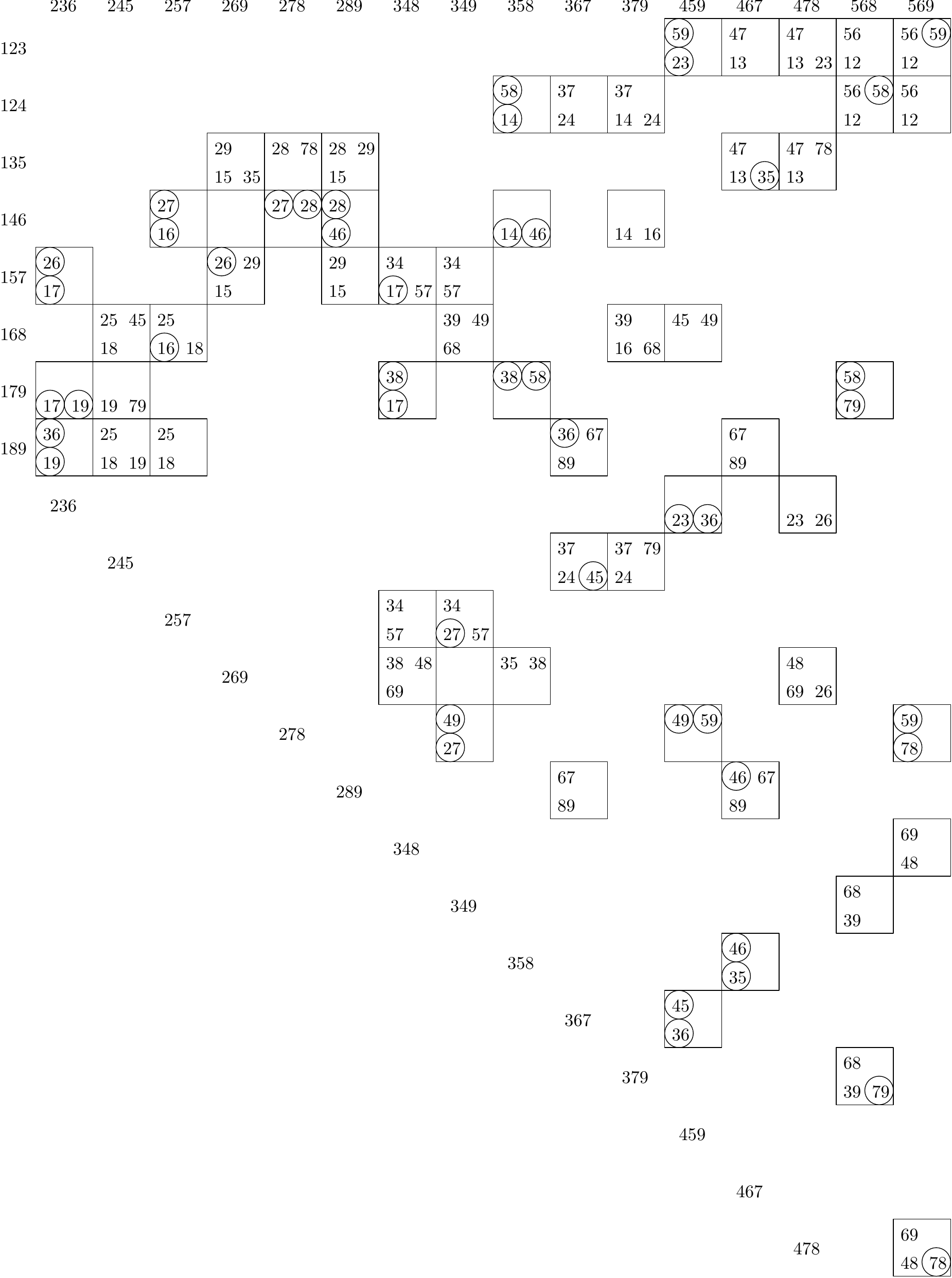}}
        \caption{Tracing the singular set in $|\Delta_1|$, second case.}
        \label{fig:8}
\end{figure}

\begin{paragraph}{Second case.}
Assume now that edges $56$, $57$, $67$ meet 
the singular set exactly once each, then so do $12$, $34$, and $89$ (the coupled edges). Figure \ref{fig:8} derives some consequences of these intersections. Note that all labels that were shaded in Figure \ref{fig:5} are now circled.

Since $\phi_{\psi}$ is supposed to be a (polyhedrally generic) polyhedral immersion, its convex hull $\mathrm{conv}(\phi_\psi(\left|\Delta_1\right|))$ is a convex polyhedron. All vertices and edges of $\mathrm{conv}(\phi_\psi(\left|\Delta_1\right|))$ are immersed edges and vertices of the neighborly triangulation $\Delta_1$. No self-intersection occurs on these edges, since that would be a contradiction to them being edges of  $\mathrm{conv}(\phi_{\psi}(\left|\Delta_1\right|))$. Therefore, the corresponding edges in $\Delta_1$ do not meet the singular set.

A convex polyhedron possesses at least four vertices, thus at least one of the immersed vertices $1$, $2$, $3$, $4$, $8$, $9$ is a vertex of the convex hull $\mathrm{conv}(\phi_{\psi}(\left|\Delta_1\right|))$. Furthermore, automorphism $\left(1 8 4 2 9 3\right)\left(5 6 7\right)$ of $\Delta_1$ exchanges these vertices. Without loss of generality (our choices so far were symmetric with respect to this automorphism)  we can assume that the (immersed) vertex $1$ is a vertex of $\mathrm{conv}(\phi_{\psi}(\left|\Delta_1\right|))$, and that at least three of the incident edges in $\Delta_1$ do not meet the singular set (as they are sent to edges of $\mathrm{conv}(\phi_{\psi}(\left|\Delta_1\right|))$ and thus cannot be involved in the self-intersection of the immersed surface). 

However, from the diagram we can conclude that besides edge $12$ (see above), edges $14$, $16$, $17$, and $19$ must meet the singular set. Furthermore, because triangle $269$ intersects triangle $157$, 
edge $15$ must pierce either triangle $269$ or $289$. Only \emph{two} edges incident to vertex $1$, $13$ and $18$, may be sent to edges of $\mathrm{conv}(\phi_{\psi}(\left|\Delta_1\right|))$. This contradiction proves that the second case cannot occur.
\end{paragraph}

Thus the $9$-vertex-triangulation $\Delta_1$ of $N_5$  is not polyhedrally immersible (realizable) in $\RR^3$.

\subsection{The second triangulation}
Let $\Delta_2$ now denote the second $9$-vertex-triangulation of $N_5$, diagrammed in Figure 
\ref{fig:9}. Assume that there exists a polyhedral realization of this triangulation, so that we may choose a polyhedrally generic immersion $\phi_{\psi}\colon\left|\Delta_2\right|\rightarrow\RR^3$. An intersection table has been prepared in Figure~\ref{fig:11}.
Due to the similar structure of both triangulations, the proof of polyhedral non-immersibility is also similar, yet for this triangulation it requires an additional argument at the end.

\begin{figure}[htbp]
        \centering
        {\includegraphics[width=0.5\linewidth]{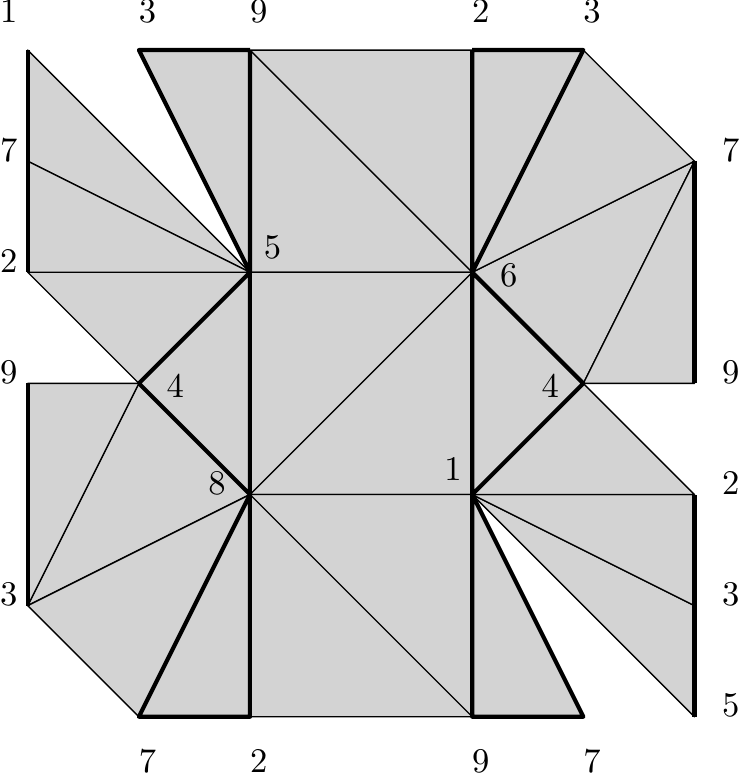}}
        \caption{The $9$-vertex-triangulation $\Delta_2$ of $N_5$.}
        \label{fig:9}
\end{figure}

\begin{figure}[htbp]
        \centering
        {\includegraphics[width=0.8\linewidth]{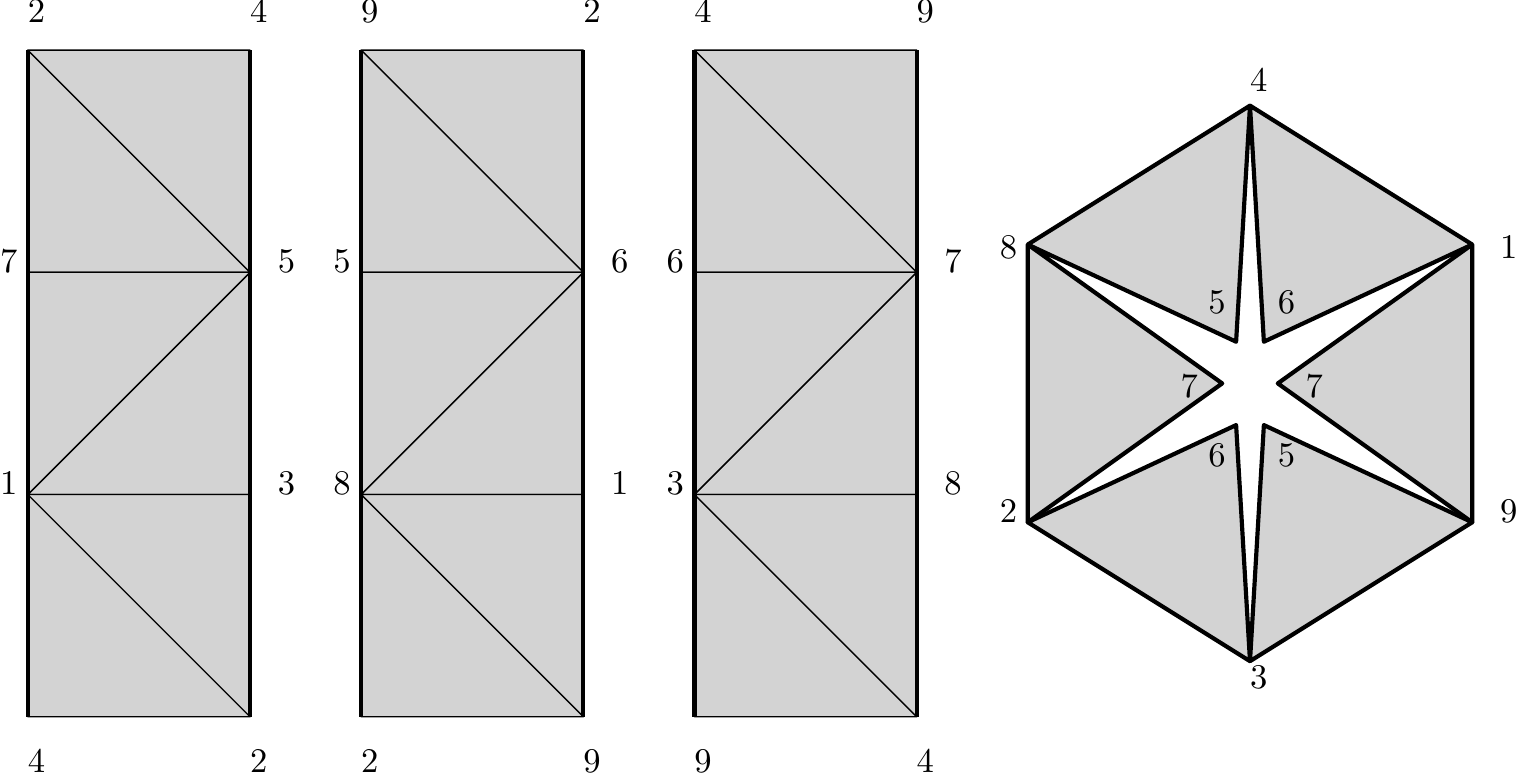}}
        \caption{$\Delta_2$ decomposes into three triangulated M\"obius strips and six remaining triangles.}
        \label{fig:10}
\end{figure}

As in Section \ref{subsec:first}, the triangulation in Figure \ref{fig:9} decomposes into three M\"obius strips and six additional triangles, see Figure \ref{fig:10}. Also, there are \emph{combinatorially coupled} edges, which can be derived from Figure~\ref{fig:11}:\\
\begin{tabular}{cccc}
$\left\{12,56\right\}$&$\left\{13,47\right\}$&$\left\{15,29\right\}$&$\left\{18,25\right\}$\\
$\left\{24,37\right\}$&$\left\{34,57\right\}$&$\left\{38,69\right\}$&$\left\{49,68\right\}$\\
$\left\{67,89\right\}$& & &\\ 
\end{tabular}

The automorphism-inducing permutation $\left(1 8 3\right)\left(5 6 7\right)\left(2 9 4\right)$ of the vertices of $\Delta_2$ exchanges the M\"obius strips in Figure \ref{fig:10}, and also exchanges the edges $24$, $29$ and $49$, which are therefore combinatorially equivalent. The curve underlying the edge cycle $24-49-29$ is orientation-reversing, and thus must meet the singular set an odd number of times (compare Lemma \ref{lem:1}). 
Similar to the situation for $\Delta_1$, Figure \ref{fig:11} reveals that every immersed edge may have intersections with at most two immersed triangles of the triangulation $\Delta_2$, as every edge label appears in precisely four boxes. In particular, each of the edges $24$, $29$ and $49$ can meet the singular set at most two times. Hence, our case distinction is between just one of $24$, $29$, $49$ intersecting the singular set precisely once, and precisely one such intersection for each of these three edges. 

\subsubsection{Case distinction}

\begin{paragraph}{First case.}
Without loss of generality (equivalence of edges $24$, $29$, $49$) assume that edge $29$ meets the singular set exactly once, whereas edges $24$ and $49$ meet the singular set exactly twice or not at all. 
Then edge $15$ (coupled with $29$) meets the singular set exactly once, whereas edges $37$ and $68$ meet the singular set exactly twice or not at all.

Consider the edge cycle $25-15-12$ in Figure \ref{fig:10}. Clearly, the curve underlying this cycle is orientation-reversing in $\left|\Delta_2\right|$. Lemma \ref{lem:1} requires an odd number of intersections of this curve with the singular set. Since edge $15$ is passed by the singular set exactly once, edges $12$ and $25$ must meet the singular set either both an odd number of times or both an even number of times. Because edge $12$ is combinatorially coupled with edge $56$, and edge $25$ is combinatorially coupled with edge $18$, edges $12$, $18$, $25$, and $56$ either all meet the singular set in an odd number of points or all meet the singular set in an even number of points. Knowing this, consider cycles $12-18-28$ and $25-56-26$, which are both orientation-preserving in $\left|\Delta_2\right|$, see Figure \ref{fig:9}. They have to meet the singular set in an even number of points by Lemma~\ref{lem:2} and, consequently, so do edges $26$ and $28$. Since edge $68$ was presumed to have two or no intersections with the singular set, this implies that the orientation-reversing cycle $26-68-28$ (see Figure \ref{fig:10}) meets the singular set in an even number of points. This contradiction to Lemma \ref{lem:1} proves that this first case cannot occur in a polyhedral immersion of $\Delta_2$.
\end{paragraph}

\begin{figure}[htbp]
        \centering
        {\includegraphics[width=1.0\linewidth]{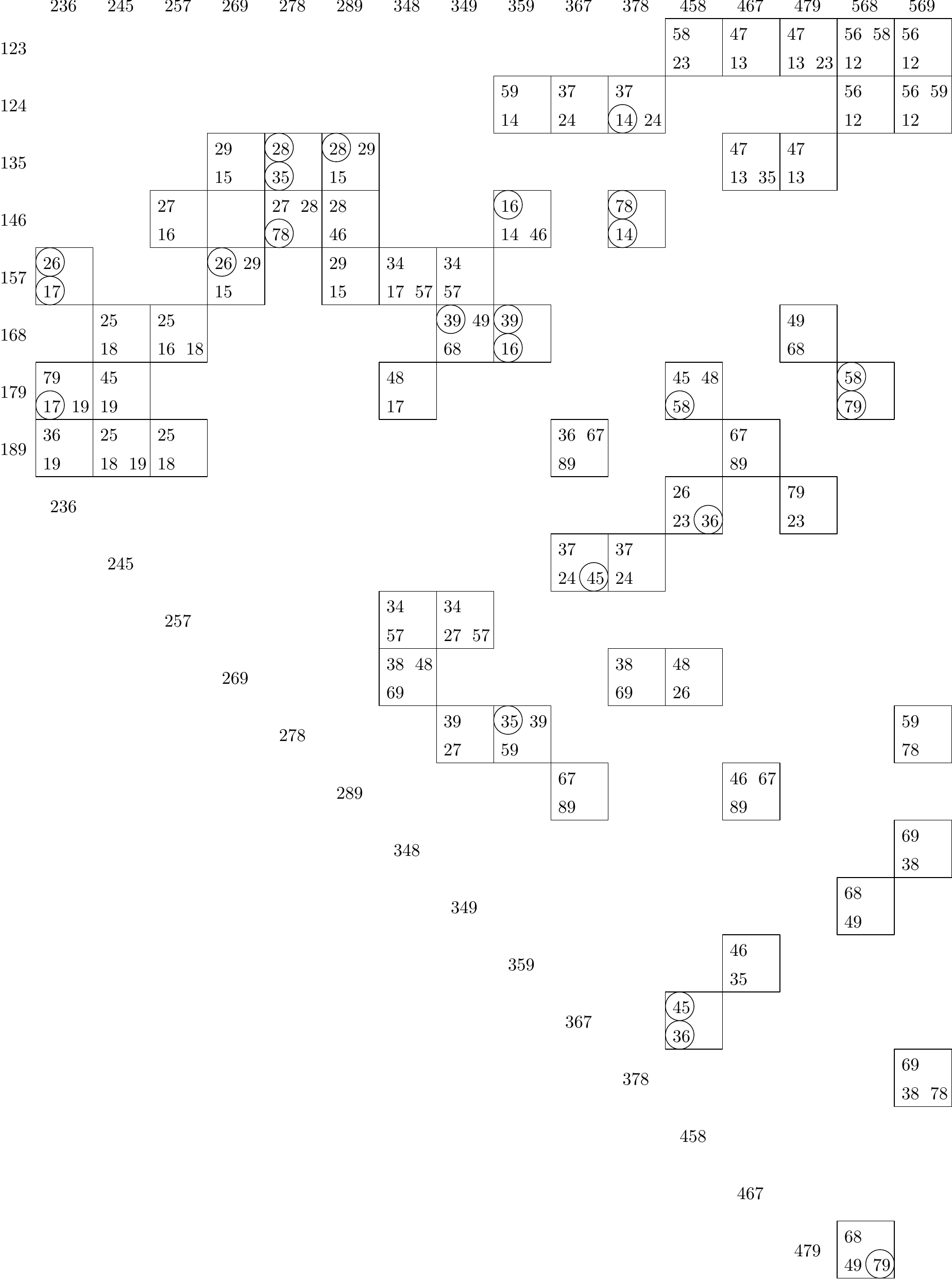}}
        \caption{Tracing the singular set in $|\Delta_2|$, second case.}
        \label{fig:11}
\end{figure}

\begin{paragraph}{Second case.}
Assume now that edges $24$, $29$, $49$ each meet the singular set exactly once, and so do their coupled edges $15$, $37$, and $68$. Figure \ref{fig:11} derives the consequences from this setup.  
 
Since $\Delta_2$ is a neighborly triangulation, all edges of the convex polyhedron $\mathrm{conv}(\phi_{\psi}(\left|\Delta_2\right|))$ are immersed edges of $\phi_{\psi}(\left|\Delta_2\right|)$. A convex polyhedron requires at least four vertices, therefore at least one of the vertices $1$, $3$, $5$, $6$, $7$, $8$ is sent to a vertex of $\mathrm{conv}(\phi_{\psi}(\left|\Delta_2\right|))$. Furthermore, the automorphism-inducing permutation $\left(1 6 3 5 8 7\right)\left(2 9 4\right)$ of the vertex set exchanges these vertices. Our previous assumptions were symmetric with respect to this automorphism, so we can now assume, without loss of generality, that (the immersed) vertex $1$ is a vertex of $\mathrm{conv}(\phi_{\psi}(\left|\Delta_2\right|))$. Then at least three edges incident to vertex $1$ in $\Delta_2$ cannot meet the singular set, as they are mapped to edges of $\mathrm{conv}(\phi_{\psi}(\left|\Delta_2\right|))$. 

We already know that $15$ meets the singular set (precisely once). From the intersection table in Figure \ref{fig:11} we conclude that edges $14$, $16$, and $17$ also meet the singular set, which means that the edges being sent to edges of $\mathrm{conv}(\phi_{\psi}(\left|\Delta_2\right|))$ incident with vertex $1$ must be among $\{12, 13, 18, 19\}$. Recall that polyhedral genericity (Definition \ref{def:1}) prohibits edges to intersect in their relative interiors. Note that this means that the double set of $\phi_{\psi}$ stays away from the boundary of $\mathrm{conv}(\phi_{\psi}(|\Delta_2|))$. At least three of the six possible cycles of immersed edges $12-23-13$, $12-28-18$, $12-29-19$, $13-38-18$, $13-39-19$, $18-89-19$ must lie on the boundary of $\mathrm{conv}(\phi_{\psi}(|\Delta_2|))$, and thus cannot have any point in common with the double set. However, edges $39$, $28$ and $29$ are met by the singular set, and only cycles $12-23-13$, $13-38-18$ and $18-89-19$ remain. Furthermore, the immersed edges $23$, $38$ and $89$ cannot form a simple closed polygonal curve connecting the neighbors of vertex $1$ on the boundary of $\mathrm{conv}(\phi_{\psi}(\left|\Delta_2\right|))$. 
This contradiction proves that the second case cannot occur.
\end{paragraph}

The preceding case distinction proves that $\Delta_2$ is not polyhedrally immersible in $\RR^3$. Since there were only two possible $9$-vertex-triangulations of $N_5$, $\Delta_1$ and $\Delta_2$, we have proved Theorem \ref{thm:main}.

\begin{theorem}\label{thm:main}
The surface $N_5$ does not admit a polyhedral realization (immersion) in $\RR^3$ with only $9$ vertices.
\end{theorem}

The number of vertices needed for a polyhedral realization of $N_5$ must be strictly greater than the minimal number of $9$ vertices needed for a triangulation, i.e., strictly greater than the Heawood bound. Do $10$ vertices suffice? This question was answered in the affirmative in the author's Diploma thesis \cite{l2009:peaiot2m}, and the result will be published in another article \cite{bl2016:peaiomt2mwfv}. In fact, even polyhedral immersions with threefold geometric symmetry exist.

\section{Remarks}\label{sec:4}

We briefly discuss possible alternative avenues for proving Theorem \ref{thm:main}. First, each of the M\"obius strips in the decompositions of $\Delta_1$ and $\Delta_2$ in Figures~\ref{fig:4} and \ref{fig:10}, respectively, must be embedded if a polyhedral realization exists. Unfortunately, this requirement alone does not lead to a quick decision of the realizability of the triangulations.

Second, a theorem of Banchoff \cite{b1974:tpasois} guarantees at least one triple point for any (polyhedrally) generic immersion of a surface of odd Euler characteristic. This fact may also serve as a starting point for investigating the course of the singular set for a purported polyhedrally generic realization of $\Delta_1$ or $\Delta_2$, ultimately resulting in a contradiction. However, pursuing this approach is more tedious than following the method presented in this article. By contrast, such an approach works well in other cases; in \cite{l2009:peaiot2m} and \cite{bl2016:peaiomt2mwfv}, it was shown that many $9$-vertex-triangulations of the projective plane and the projective plane with one handle are incompatible with triple points, rendering them non-realizable in our context. 

The purpose of this article was to demonstrate that settling the realizability question by hand, in particular proving polyhedral non-realizability, is possible in certain instances. The non-realizable $9$-vertex triangulations of $N_5$ of Theorem \ref{thm:main} are from an interesting class, namely neighborly triangulations, and all of them could be excluded due to their very limited number.  Further non-realizability results have been obtained by variants of the method presented in this article, e.g., for additional triangulated Klein bottles and projective planes in \cite{l2009:peaiot2m}. Naturally, the small number of vertices and the high degree of connectivity between the vertices play an important role for the feasibility of the method, which is only suitable for non-orientable surfaces as it relies upon tracing the self-intersection. 

Remarkably, for orientable surfaces, the tetrahedron (genus $0$) and the Cs\'asz\'ar torus \cite{c1949:apwd} with $7$ vertices (genus $1$) are embedded polyhedra which are both realized from neighborly triangulations. Yet the next class of neighborly orientable triangulations, in terms of number of vertices, does not admit a single such realization; none of the $59$ triangulations of the orientable surface of genus $6$ with $12$ vertices are realizable, see \cite{s2010:nmvtossnuomass}.

For non-orientable neighborly triangulations, Brehm \cite{b1983:antms} noted that $6$ vertices do not suffice to create the required triple point in a polyhedral realization of a projective plane (see also the discussion on triple points above). In this paper, we have now demonstrated non-realizability for the neighborly $9$-vertex-triangulations of $N_5$. It would be interesting to know what the situation is for the fourteen neighborly $10$-vertex triangulations of~$N_7$ which are listed in \cite{lutz:manifoldpage}.

\section{Acknowledgments}

The author would like to thank Ulrich Brehm for the support and guidance received while working on her Diploma Thesis at TU Dresden under his supervision, as well as for the continuing stimulating discussions on the subject. 


\begin{thebibliography}{10}

\bibitem{abem2007:htetmis}
Dan Archdeacon, C.~Paul Bonnington, and Joanna~A. Ellis-Monaghan.
\newblock How to exhibit toroidal maps in space.
\newblock {\em Discrete Comput. Geom.}, 38(3):573--594, 2007.

\bibitem{b1974:tpasois}
Thomas~F. Banchoff.
\newblock Triple points and surgery of immersed surfaces.
\newblock {\em Proc. Amer. Math. Soc.}, 46:407--413, 1974.

\bibitem{bb1987:anpog3w10v}
J.~Bokowski and U.~Brehm.
\newblock A new polyhedron of genus {$3$} with {$10$} vertices.
\newblock In {\em Intuitive geometry ({S}i\'ofok, 1985)}, volume~48 of {\em
  Colloq. Math. Soc. J\'anos Bolyai}, pages 105--116. North-Holland, Amsterdam,
  1987.

\bibitem{bgo2000:otgoom}
J.~Bokowski and A.~Guedes~de Oliveira.
\newblock On the generation of oriented matroids.
\newblock {\em Discrete Comput. Geom.}, 24(2-3):197--208, 2000.
\newblock The Branko Gr{\"u}nbaum birthday issue.

\bibitem{b1989:agrwsidefdrm}
J{\"u}rgen Bokowski.
\newblock A geometric realization without self-intersections does exist for
  {D}yck's regular map.
\newblock {\em Discrete Comput. Geom.}, 4(6):583--589, 1989.

\bibitem{b2008:ohmffros}
J{\"u}rgen Bokowski.
\newblock On heuristic methods for finding realizations of surfaces.
\newblock In {\em Discrete differential geometry}, volume~38 of {\em
  Oberwolfach Semin.}, pages 255--260. Birkh\"auser, Basel, 2008.

\bibitem{bb1989:apog4wmnovams}
J{\"u}rgen Bokowski and Ulrich Brehm.
\newblock A polyhedron of genus {$4$} with minimal number of vertices and
  maximal symmetry.
\newblock {\em Geom. Dedicata}, 29(1):53--64, 1989.

\bibitem{bs1989:csg}
J{\"u}rgen Bokowski and Bernd Sturmfels.
\newblock {\em Computational {S}ynthetic {G}eometry}, volume 1355 of {\em
  Lecture Notes in Mathematics}.
\newblock Springer-Verlag, Berlin, 1989.

\bibitem{bl2016:peaiomt2mwfv}
U.~Brehm and U.~Leopold.
\newblock Polyhedral {E}mbeddings and {I}mmersions of {M}any {T}riangulated
  $2$-{M}anifolds with {F}ew {V}ertices.
\newblock arXiv:1603.04877, 2016.

\bibitem{b1981:pmzevgd}
Ulrich Brehm.
\newblock Polyeder mit zehn {E}cken vom {G}eschlecht drei.
\newblock {\em Geom. Dedicata}, 11(1):119--124, 1981.

\bibitem{b1983:antms}
Ulrich Brehm.
\newblock A nonpolyhedral triangulated {M}\"obius strip.
\newblock {\em Proc. Amer. Math. Soc.}, 89(3):519--522, 1983.

\bibitem{b1987:msprodrm}
Ulrich Brehm.
\newblock Maximally symmetric polyhedral realizations of {D}yck's regular map.
\newblock {\em Mathematika}, 34(2):229--236, 1987.

\bibitem{b1987:amspog3w10v}
Ulrich Brehm.
\newblock A maximally symmetric polyhedron of genus {$3$} with {$10$} vertices.
\newblock {\em Mathematika}, 34(2):237--242, 1987.

\bibitem{b1990:htbmpmotbs}
Ulrich Brehm.
\newblock How to build minimal polyhedral models of the {B}oy surface.
\newblock {\em Math. Intelligencer}, 12(4):51--56, 1990.

\bibitem{bs1997:pm}
Ulrich Brehm and Egon Schulte.
\newblock Polyhedral maps.
\newblock In {\em Handbook of discrete and computational geometry}, CRC Press
  Ser. Discrete Math. Appl., pages 345--358. CRC, Boca Raton, FL, 1997.

\bibitem{c1994:vmsiotkbits}
Davide~P. Cervone.
\newblock Vertex-minimal simplicial immersions of the {K}lein bottle in three
  space.
\newblock {\em Geom. Dedicata}, 50(2):117--141, 1994.

\bibitem{c1949:apwd}
{\'A}kos Cs{\'a}sz{\'a}r.
\newblock A polyhedron without diagonals.
\newblock {\em Acta Univ. Szeged. Sect. Sci. Math.}, 13:140--142, 1949.

\bibitem{h1890:mct}
P.J. Heawood.
\newblock {M}ap {C}olour {T}heorem.
\newblock {\em Quart. J. Math.}, 24:332--338, 1890.

\bibitem{hlz2010:srwtisf}
Stefan Hougardy, Frank~H. Lutz, and Mariano Zelke.
\newblock Surface realization with the intersection segment functional.
\newblock {\em Experiment. Math.}, 19(1):79--92, 2010.

\bibitem{jr1980:mtoos}
M.~Jungerman and G.~Ringel.
\newblock Minimal triangulations on orientable surfaces.
\newblock {\em Acta Math.}, 145(1-2):121--154, 1980.

\bibitem{l2009:peaiot2m}
Undine Leopold.
\newblock {\em {P}olyhedral {E}mbeddings and {I}mmersions of {T}riangulated
  2-{M}anifolds}.
\newblock Diploma Thesis, TU Dresden, 2009.

\bibitem{l2008:earrots}
Frank~H. Lutz.
\newblock Enumeration and random realization of triangulated surfaces.
\newblock In {\em Discrete differential geometry}, volume~38 of {\em
  Oberwolfach Semin.}, pages 235--253. Birkh\"auser, Basel, 2008.

\bibitem{lutz:manifoldpage}
Frank~H. Lutz.
\newblock {T}he {M}anifold {P}age.
\newblock {TU} {B}erlin:
  {\url{http://www.math.tu-berlin.de/diskregeom/stellar/}}, accessed May 2016.

\bibitem{mcmsw1983:p2mie3wulg}
P.~McMullen, Ch. Schulz, and J.~M. Wills.
\newblock Polyhedral {$2$}-manifolds in {$E^{3}$} with unusually large genus.
\newblock {\em Israel J. Math.}, 46(1-2):127--144, 1983.

\bibitem{r1955:wmdgnfimwdzk}
Gerhard Ringel.
\newblock Wie man die geschlossenen nichtorientierbaren {F}l\"achen in
  m\"oglichst wenig {D}reiecke zerlegen kann.
\newblock {\em Math. Ann.}, 130:317--326, 1955.

\bibitem{s2010:nmvtossnuomass}
Lars Schewe.
\newblock Nonrealizable minimal vertex triangulations of surfaces: showing
  nonrealizability using oriented matroids and satisfiability solvers.
\newblock {\em Discrete Comput. Geom.}, 43(2):289--302, 2010.

\bibitem{s1922:pur}
Ernst Steinitz.
\newblock {P}olyeder und {R}aumeinteilungen.
\newblock {\em Encyclop{\"a}die der mathematischen {W}issenschaften}, 3
  (Geometrie):1--139, 1922.

\bibitem{z1995:lop}
G{\"u}nter~M. Ziegler.
\newblock {\em Lectures on {P}olytopes}, volume 152 of {\em Graduate Texts in
  Mathematics}.
\newblock Springer-Verlag, New York, 1995.

\end{thebibliography}
\end{document}